\documentclass{article}
\usepackage{amsmath}
\usepackage{graphicx}
\usepackage{amssymb}
\usepackage{amsthm}
\usepackage[english]{babel}

\newcommand{\lrb}[1]{\left [ #1 \right ] }
\newcommand{\lrbrace}[1]{\left \lbrace #1 \right \rbrace }
\newcommand{\lrrb}[1]{\left ( #1 \right ) }

\newcommand{\abs}[1]{\left| #1 \right|}

\newcommand{\set}[1]{\mathbb{#1}}

\newcommand{\nelmt}[1]{\; \epsilon \! \! \! / \; #1}

\newtheorem*{appendixtheorem}{Theorem}
\newtheorem{maintheorem}{Theorem}
\newtheorem{theorem}{Theorem}[section]
  \newtheorem{definition}[theorem]{Definition}
  \newtheorem{lemma}[theorem]{Lemma}
  \newtheorem{corollary}[theorem]{Corollary}
  \newtheorem{remark}[theorem]{Remark}

\begin{document}

\author{Sebastian Helmensdorfer}
\title{A model for the behaviour of fluid droplets based on mean curvature flow}
\maketitle

\begin{abstract}
The authors of \cite{naturephys} have observed the following remarkable phenomenon during their experiments. If two oppositely charged droplets of fluid are close  enough, at first they attract each other and touch eventually. Surprisingly after that the droplets are repelled from each other, if the initial strength of the charges is high enough. Otherwise they coalesce and form a big drop, as one might expect. 

We present a theoretical model for these observations, using mean curvature flow. The local asymptotic shape of the touching fluid droplets is that of a double cone, where the angle corresponds to the strength of the initial charges. Our model yields  a critical angle for the behaviour of the touching droplets and numerical estimates of this angle agree with the experiments. This shows, contrary to  general belief (see \cite{naturephys} and \cite{physrevletters}), that decreasing surface energy can explain the phenomenon.

To determine the critical angle within our model we construct appropriate barriers for the mean curvature flow. In \cite{angenentilmanen} Angenent, Chopp and Ilmanen manage to show the existence of one-sheeted and two-sheeted, self-expanding solutions with a sufficiently steep double cone as an initial condition. Furthermore they provide arguments for nonuniqueness even among the one-sheeted solutions. We present a proof for this, yielding a slightly stronger result. Using the one-sheeted self-expanders as barriers we can determine the critical angle for our model.
\end{abstract}




\section{Introduction}

In \cite{naturephys} the behaviour of oppositely charged droplets of fluid is investigated. Droplet motion induced by electrical charges occurs in a vast 
number of applications, including storm cloud formation, commercial ink-jet printing, petroleum and vegetable oil dehydration, electrospray ionization for 
use in mass spectrometry, electrowetting and lab-on-a-chip manipulations (see also \cite{naturephys}).

The phenomenon can be described as follows. Two close enough oppositely charged droplets of fluid attract each other and converge. Both, experiments and 
numerical simulations (see \cite{physrevletters}), provide evidence that a short-lived bridge is formed between the droplets, which instantly causes the 
charges to be exchanged. The bridge between the touching droplets has the local asymptotic shape of a double cone. Furthermore the charges determine the 
angle of the double cone, where a lower charge corresponds to a steeper cone (i.e. having a larger acute angle with the rotation axis). 

We want to study the behaviour of the system after the droplets have touched. One might think that coalescence occurs and that one big drop is formed. 
However experiments in \cite{naturephys} have shown that above a critical field strength the droplets do not coalesce after touching but are repelled from 
each other. 

Following an idea by P. Topping, we present a theoretical model for this phenomenon. It is assumed that after the two droplets have touched and exchanged 
their charges, the motion of the system is driven by minimization of energy. To model this mathematically we use the mean curvature flow which is the 
gradient flow of the surface area (see e.g. \cite{eckermcf}). Our model has two main advantages over the theoretical approaches presented in \cite{naturephys}
 and \cite{physrevletters}. Firstly we are able to show that minimization of surface energy can explain the observations from the experiments. Secondly 
our results are mostly independent from any assumption on the precise local conical shape formed by the touching droplets.

We define the double cone in $\set{R}^n$ for $0 < \alpha < \frac{\pi}{2}$ as
$$D_\alpha = \left\lbrace x \in \set{R}^n:\; \abs{\left(x_2,\ldots,x_n \right)}^2 = \tan^2 \alpha \cdot x_1^2 \right\rbrace .$$ 
The results we obtain from our model can be summarized as follows.

\emph{
Assume that two initially oppositely charged droplets of fluid after touching have the local shape of a smoothing of a double cone $ D_\alpha $ in 
$\set{R}^3$. Also assume that their motion is governed by minimization of area which we model using the mean curvature flow.
}

\emph{
Then there is a critical angle $ \alpha_{crit} $ with the following properties. If $ \alpha < \alpha_{crit}$ the droplets are repelled from each 
other. If the associated smoothing lies outside of $ D_\alpha $ and $ \alpha > \alpha_{crit}$ the droplets coalesce and form one big drop.
}

\emph{
Using appropriate barriers and a level-set flow argument we can conclude that $\alpha_{crit}$ is precisely the critical angle for the existence of one-sheeted, self-expanding evolutions of $D_\alpha \subset \set{R}^3$, which means $\alpha_{crit} \approx 66^\circ$.
}

These formulations are made precise in the sequel. The critical angle of $60^\circ-70^\circ$, observed during experiments (see \cite{naturephys}), agrees with our prediction.

A family of smooth, immersed hypersurfaces $ \left(M_t\right)_{t \in  I} $ ($ I $ a real interval) in $ \set{R}^n $ is called a solution of the mean 
curvature flow if
\begin{equation}
\dfrac{\partial x}{\partial t} = \mathbf{H}, \; \; \; \; x \in M_t,\; t \in I. 	\label{eq:MCF}
\end{equation}
$ \mathbf{H} = -H \upsilon$ is the mean curvature vector and $ \upsilon $ a choice of unit normal. (\ref{eq:MCF}) is equivalent to 
$ \frac{\partial x}{\partial t} = \triangle_{M_t}x $ ($\triangle_{M_t}$ is the Laplacian of $M_t$).

Based on the ideas in \cite{angenentilmanen} we present a new proof for the existence of one-sheeted self-expanders with the double cone  $D_\alpha$ for $\alpha$ large enough as an initial condition. We call a solution of (\ref{eq:MCF}) self-expanding if
\begin{equation}
M_t = \sqrt{t} \cdot M_1, \; \; \; \; t \in (0,\infty). \label{eq:selfexp}
\end{equation}
The singular initial condition $ D_\alpha $ is here understood to be attained locally in the sense of Hausdorff distance. 

\begin{maintheorem}
\label{thm:selfexp}
For $ n \geq 3 $ there exists a critical angle $ \alpha_{crit}^*(n) \in \lrrb{0,\frac{\pi}{2}}$ with the following properties. 

For any angle $ \alpha > \alpha_{crit}^* $ there exist at least three distinct, smooth, rotationally symmetric evolutions of the double cone 
$ D_\alpha $ which are self-expanding. Two of these evolutions are one-sheeted and one is two-sheeted. 

For $\alpha = \alpha_{crit}^*$ at least one one-sheeted and one two-sheeted self-expanding, smooth, rotationally symmetric evolution of $ D_\alpha $ exist.
\end{maintheorem}

Here ''one-sheeted'' and ''two-sheeted'' refer to the number of connected components of the solutions. The existence of one-sheeted self-expanders was first proved in \cite{angenentilmanen}. Additionally our proof provides nonuniqueness among one-sheeted solutions, which was also first stated in \cite{angenentilmanen}.

Nonuniqueness here corresponds to fattening of the level-set flow. For $\alpha < \alpha_{crit}^*$ the two-sheeted evolution of $D_\alpha$ is unique and therefore we have non-fattening of the level-set flow (see \cite{angenentilmanen}).

Using the one-sheeted self-expanders we can show that $\alpha_{crit} = \alpha_{crit}^*(3)$.

\section{Existence of self-expanders}

In this section we want to study the evolution by mean curvature of the double cone $ D_\alpha $, following \cite{angenentilmanen}. We are interested in 
solutions $ \left(M_t \right)_{t \in (0,\infty)} $ that satisfy
\begin{equation}
M_t \text{ is rotationally symmetric and self-expanding.} \label{eq:rotselfexp}
\end{equation}
From (\ref{eq:MCF}) and (\ref{eq:selfexp}) one can see that for a solution of the mean curvature flow to move self-expanding is equivalent to
the self-expanding equation
\begin{equation}
H = -\dfrac{x\cdot \upsilon}{2}, \; \; \; \;  x \in M_1. \label{eq:selfexp1}
\end{equation}
Under condition (\ref{eq:rotselfexp}) this equation becomes an ODE.

We write $ x = \left( x_1,x_2,\ldots,x_n \right) = \left(x_1,\hat{x} \right) $ for $ x \in \set{R}^n = \set{R}\times\set{R}^{n-1} $. For a curve 
$ \gamma $ in $ \set{R}^2 $, that is symmetric with respect to $ u \mapsto -u $ we define the corresponding surface of rotation 
$$ M \left( \gamma \right) =  \left\lbrace \left(x_1,\hat{x} \right) \in \set{R}^n: \; \left(x_1,\abs{\hat{x}} \right) = \left(y,u \right) \in \gamma \right\rbrace.$$
For $ -\pi \leq \alpha \leq \pi $ let $ \sigma_\alpha $ be the closed ray $ \left\lbrace t (\cos \alpha, \sin \alpha): \; t \geq 0 \right\rbrace $ in 
$ \set{R}^2 $. The cone in $ \set{R}^n $ with angle $\alpha$ is defined as $ C_\alpha = M \left( \sigma_\alpha \cup \sigma_{-\alpha} \right) $. We 
get $ D_\alpha = C_\alpha \cup C_{\pi - \alpha} $. 

Following \cite{angenentilmanen} we get an equation for $ \gamma $ from (\ref{eq:rotselfexp}). Parametrizing $ \gamma $ by arclength $ s $ we define 
$ \theta \in [0,2\pi) $ along $ \gamma $ by setting $\gamma_s = (\cos \theta, \sin \theta)$ for the tangent vector $\gamma_s$. The left-handed unit normal is given by 
$ \overline{\upsilon} = (-\sin \theta, \cos \theta ) $.  Let $ \mathbf{k} $ be the curvature vector of $ \gamma $ and 
$ k = \mathbf{k} \cdot \overline{\upsilon} $. Equation (\ref{eq:rotselfexp}) for $ M \left( \gamma \right) $ becomes
\begin{equation}
k - \dfrac{n-2}{u} \cos \theta + \dfrac{y}{2} \sin \theta - \dfrac{u}{2} \cos \theta = 0, \; \; \; \; \left(y,u\right) \in \gamma . \label{eq:rotsymmexpcurve}
\end{equation}
A solution of this equation creates a smooth surface $ M\left( \gamma \right) $.

The following lemma is due to Angenent, Chopp and Ilmanen (see \cite{angenentilmanen}). The first part is a consequence of results from Ecker and Huisken on graphical mean curvature flow (see \cite{eckerhuisken}).

\begin{lemma} \label{lem:ilmanen1}
(i) (Two-sheeted case)

For $\alpha \in (0,\pi)$ there exists a unique, smooth, connected curve $ \gamma(\alpha) $, solving (\ref{eq:rotsymmexpcurve}) and asymptotic to $C_\alpha$. Furthermore, unless $\gamma$ is the $u$-axis, $\gamma$ is the graph of a positive, convex (or negative, concave) even function $ y=y(u) $.

(ii) (One-sheeted case)

Let $\gamma$ be another smooth, connected curve solving (\ref{eq:rotsymmexpcurve}) which meets $ \left\lbrace u > 0 \right\rbrace $. Then $\gamma$ lies in 
$ \left\lbrace u > 0 \right\rbrace $ and is asymptotic to $ \sigma_\alpha \cup \sigma_\beta $, where $ 0<\alpha<\beta<\pi $. If $ \gamma $ meets the 
$ u $-axis at a right angle, then $ \gamma $ is the graph of a positive, even function $ u=u\left( y \right) $ which is monotone for $ y \neq 0 $ and 
$ \beta = \pi - \alpha $.
\end{lemma}

We define $$ A = \left\lbrace \alpha \in \left(0,\frac{\pi}{2} \right): \; \exists \text{ connected } \gamma \text{ solving  (\ref{eq:rotselfexp}), asymptotic to } \sigma_\alpha \cup \sigma_{\pi-\alpha} \right\rbrace $$ and $ \alpha_{crit}^*(n) = \inf A $.

Now we want to focus on smooth, connected curves $ \gamma $ solving (\ref{eq:rotsymmexpcurve}) which are graphs of even functions (part (ii) of 
lemma \ref{lem:ilmanen1}). After imposing initial conditions we get the following initial value problem for $ u $
\begin{equation}
u_{yy} = \left(1+\left(u_y \right)^2 \right) \left( \frac{1}{2} \left(u-u_y y \right) + \frac{n-2}{u} \right), \; \; \; \; u(0)=C>0,\; u_y(0)=0. \label{eq:selfexprotcurgraph}
\end{equation}

For symmetry reasons it is enough to consider $ u \vert _{[0,\infty)} $ which  we denote again by $ u $. For $ y > 0 $ we denote by $ \alpha(y) = \arctan \left(\frac{u}{y} \right) $ the signed angle, that $ u $ makes with the positive  $ y $-axis. First we note the following basic observations regarding the solution $u$.

\begin{enumerate}
\renewcommand{\theenumi}{\roman{enumi}}

\item Every critical point of $ u $ is a strict local minimum.\\ \label{prop1}
\small {To see this compute $ u_{yy} = \frac{u}{2} + \frac{n-2}{u} > 0 $ from (\ref{eq:selfexprotcurgraph}) whenever $ u_y = 0 $.}

\item $ u_y > 0 $ for all $ y > 0 $ and every critical point of $ u_y $ is a strict local maximum.\\ \label{prop2}
\small {The first part is a direct consequence of \ref{prop1} and the fact that $ u_{yy}(0) = \frac{C}{2} + \frac{n-2}{C} > 0 $. For the second part compute 
$ u_{y}^{(3)} = -\left(1+\left(u_y \right)^2 \right) (n-2) \frac{u_y}{u^2} < 0 $ from (\ref{eq:selfexprotcurgraph}) using the first part  whenever 
$ u_{yy} = 0 $.}

\item Every critical point of $ \alpha \left( y \right) $ is a strict local minimum. Therefore $ u $ is asymptotic to a ray 
$ \sigma_\alpha $, $ \alpha =  \lim_{y \rightarrow \infty} \alpha \left( y \right)$.\\ \label{prop3}
\small {Let $ 0= \alpha_y = \frac{1}{u^2+y^2} \left(yu_y-u \right), \; y > 0$. Using (\ref{eq:selfexprotcurgraph}) we get 
$ 0=yu_y-u= \frac{2(n-2)}{u}-\frac{2u_{yy}}{1+ \left(u_y\right)^2}$ and therefore $ u_{yy} >0 $. But then $ \alpha_{yy} = \frac{yu_{yy}}{u^2+y^2}>0 $.}

\item $ 0 <  \alpha =  \lim_{y \rightarrow \infty} \alpha \left( y \right) <  \frac{\pi}{2}$. \\ \label{prop4}
\small {This is poved in \cite{angenentilmanen} using the clearing out lemma (see appendix \ref{appendixclearingout}).}

\item $u_{y} \rightarrow \tan \alpha $ as $ y \rightarrow \infty $ and hence $ \abs{u_y} $ is bounded. \\ \label{prop5}
\small {Since $ u_{yy}(0) > 0 $, \ref{prop2} shows that $ u_{y} $ is either strictly monotone increasing or has one strict local maximum. Therefore 
\ref{prop3} and \ref{prop4} yield the claim. }

\item $u_{yy} $ has at most one zero. Therefore $u_{yy} \rightarrow 0$ as $ y \rightarrow \infty $. \\ \label{prop6}
\small { This follows from \ref{prop2} and \ref{prop5}.}

\end{enumerate}

To prove theorem \ref{thm:selfexp} we show first that the asymptotic angle $ \lim_{y\rightarrow \infty}  \alpha \left( y \right)  $ of $ u $ goes to $\frac{\pi}{2}$, when $ C $ goes to $ 0 $ or $ \infty $ in (\ref{eq:selfexprotcurgraph}). Together with continuous dependence of the solutions $u$ on $C > 0$ this is enough to prove theorem \ref{thm:selfexp}.

As a first step we show that the first derivative of the solutions $u$ blows up, when $ C $ goes to $ 0 $ or $ \infty $ in (\ref{eq:selfexprotcurgraph}). In view of \ref{prop5} this is useful for investigating the behaviour of the corresponding asymptotic angles.

\begin{lemma}
\label{lem:derblowup}
 The solutions $u$ of (\ref{eq:selfexprotcurgraph}) satisfy $\sup_{[0,\infty)} u_y \rightarrow \infty$ if $C \rightarrow 0$ or $C \rightarrow \infty$.
\end{lemma}
\begin{proof}
 First we treat the case $C \rightarrow \infty$. We can write (\ref{eq:selfexprotcurgraph}) as
$$\frac{n-2}{u} + \frac{u}{2} = \frac{yu_y}{2} + \frac{u_{yy}}{1+\lrrb{u_y}^2}.$$
Since $u(0) = C$, \ref{prop2} implies that $u \lrrb{y_0} \rightarrow \infty$ as $C \rightarrow \infty$ for any fixed $y_0 > 0$. Therefore the last equation
yields either $u_y \lrrb{y_0} \rightarrow \infty$ as $C \rightarrow \infty$ or $u_{yy} \lrrb{y_0} \rightarrow \infty$ as $C \rightarrow \infty$. In both cases we can deduce the claim.

For the case $C \rightarrow 0$ suppose not, so we may assume without loss of generality that $\sup_{[0,\infty)} u_y \leq D, \; D > 0$ for all $C > 0$ small enough. First we use the clearing out lemma (see appendix \ref{appendixclearingout}) to show that $u$ must be large compared to such $C > 0$ away from $0$. Therefore assume that for $C > 0$ small enough
$$ u \lrrb{LC} \leq K C, \; \; \; \; K > 0, \; L > 0 $$
for a constant $K$, where $L$ will be determined later. 

Now we want to apply the clearing out lemma (see appendix \ref{appendixclearingout}) to the self-expander $M(u)$ (considering $u$ as a function on $\set{R}$ here) with respect to the ball $B_{LC} \lrrb{0,C,0,\ldots,0} \subset \set{R}^n$. Estimating the Hausdorff measure yields
$$ \mathcal{H}^{n-1} \lrrb{M(u) \cap B_{LC} \lrrb{0,C,0,\ldots,0}} \leq D K^{n-2} L C^{n-1}$$
where now $D = D(n)$. Now choose $L\lrrb{D,K,n} > 0$, $L = K \lrrb{\frac{D}{\varepsilon_0}}^{\frac{1}{n-2}}$, to make $D K^{n-2} L C^{n-1} \leq \varepsilon_0 L^{n-1} C^{n-1}$, where $\varepsilon_0$ is the constant from the clearing out lemma.

Therefore the clearing out lemma says that the clearing out time $t_C$ of $M(u)$ with respect to $B_{LC/4} \lrrb{0,C,0,\ldots,0}$ can be estimated as $t_C \leq c L^2C^2$ for some constant $c > 0$. But since $M(u)$ moves self-expanding we also have $ \sqrt{1+t_C}C \geq C + \frac{LC}{4}$ and hence $t_C \geq \frac{L^2}{4}$ - a contradiction.

This means $u \lrrb{LC}$ must be large for any $C > 0$ small enough. More precisely there must be a sequence $\lrrb{C_k}_{k \in \set{N}}$ which satisfies
$$ \lim_{k \rightarrow \infty} LC_k = 0, \; u_k \lrrb{LC_k} \geq k C_k .$$
where $u_k$ are the associated solutions of (\ref{eq:selfexprotcurgraph}) with $u_k(0) = C_k$ and $L(D,k,n)$ as above. By \ref{prop2} we have 
$$ u_k \lrrb{LC_k} = C_k + \int_0^{LC_k} \lrrb{u_k}_{\tilde{y}} d \tilde{y} \leq C_k + LC_k \max_{\lrb{0,LC_k}} \lrrb{u_k}_y $$
and therefore $\max_{\lrb{0,LC_k}} \lrrb{u_k}_y \geq \frac{1}{D}, \; D = D\lrrb{n} > 0$.

Now we use continuous dependence to get contradiction to the last estimate. For any $k_0 \in \set{N}$ we can choose $y_{k_0} > 0$ sufficiently small to make $\lrrb{u_{k_0}}_y < \frac{1}{2D}$ on $\lrb{0,y_{k_0}}$. Since $u_y$ depends continuously on the initial values of (\ref{eq:selfexprotcurgraph}) we can choose $k_0 \in \set{N}$ such that $\max_{\lrb{0,y_{k_0}}} \abs{\lrrb{u_{k_0}}_y - \lrrb{u_k}_y} < \frac{1}{2D}$ for $k \geq k_0$. Therefore we get 
$$\abs{\lrrb{u_k}_y} \leq \abs{\lrrb{u_{k_0}}_y - \lrrb{u_k}_y} + \abs{\lrrb{u_{k_0}}_y} < \frac{1}{D}$$
on $\lrb{0,y_{k_0}}$ for any $k \geq k_0$ - a contradiction.

\end{proof}

From the last lemma we know that $u_y$ blows up somewhere on $\left( 0,\infty \right)$ as $C \rightarrow 0$ or $C \rightarrow \infty$. To get the desired behaviour of the asymptotic angles, it is therefore enough to show that $u_y$ does not decrease too much after a possible maximum. 

\begin{lemma}
For solutions $ u $ of (\ref{eq:selfexprotcurgraph}) we have $ \alpha \rightarrow \frac{\pi}{2} $ as $ C \rightarrow 0 $ or $ C \rightarrow \infty $.
\end{lemma}
\begin{proof}
Because of \ref{prop6} and lemma \ref{lem:derblowup} we may assume that for any $C > 0$ there is precisely one zero of $u_{yy}$, $\hat{y}_C > 0$. We have from (\ref{eq:selfexprotcurgraph}) for $y > 0$
\begin{eqnarray}
\nonumber \frac{u}{y} = -\frac{2(n-2)}{uy} + u_y + \frac{2 u_{yy}}{y \lrrb{1 + \lrrb{u_y}^2}} \\
\nonumber \lrrb{\frac{u}{y}}_y = \frac{u_y}{y} - \frac{u}{y^2} = \frac{2}{y^2} 
\lrrb{\frac{n-2}{u}-\frac{u_{yy}}{1+\lrrb{u_y}^2}}.
\end{eqnarray}

This means $\lrrb{\frac{u}{y}}_y > 0$ on $\left [ \hat{y}_C, \infty \right )$ for any $C > 0$. By \ref{prop3} we have $\frac{u}{y} \rightarrow \alpha$ as $y \rightarrow \infty$, so it is enough to show $\frac{u }{y} \rightarrow \infty$ at some point in $\left[ \hat{y}_C, \infty \right)$  if either $C \rightarrow \infty$ or $C \rightarrow 0$.

We treat the case $C \rightarrow \infty$ first. Suppose not, so without loss of generality there exists $D > 0$ such that $\frac{u\lrrb{\hat{y}_C}}{\hat{y}_C} \leq D$ for any $C > 0$ large enough. Hence by the above equation $-\frac{2(n-2)}{u\lrrb{\hat{y}_C}\hat{y}_C} + u_y\lrrb{\hat{y}_C} \leq D$ and by lemma \ref{lem:derblowup} $u_y \lrrb{\hat{y}_C} \rightarrow \infty$ as $C \rightarrow \infty$. Therefore we have $u \lrrb{\hat{y}_C} \hat{y}_C \rightarrow 0$ as $C \rightarrow \infty$. But by \ref{prop2} we have $u \geq C$ on $[0,\infty)$, which
then yields $\frac{u \lrrb{\hat{y}_C}}{\hat{y}_C} \geq C^2$ - a contradiction.

For the case $C \rightarrow 0$ again suppose not. Without loss of generality we have then $\frac{u}{\hat{y}} \leq D$ on $\left[ \hat{y}_C, \infty \right)$ for some $D > 0$ and for any $C > 0$ small enough. 

For such $C > 0$ we get $\frac{u \lrrb{2 \hat{y}_C}}{2 \hat{y}_C} \leq D$ and
$$ u \lrrb{2 \hat{y}_C} = u \lrrb{\hat{y}_C} + \int_{\hat{y}_C}^{2\hat{y}_C} u_{\tilde{y}} d \tilde{y} \geq \hat{y}_C u_y \lrrb{2\hat{y}_C}.$$
Hence $u_y \lrrb{2 \hat{y}_C} \leq 2D$. On the other hand (\ref{eq:selfexprotcurgraph}) written as above in the proof yields
$$ 0 < \frac{u \lrrb{2 \hat{y}_C}}{2\hat{y}_C} \leq - \frac{2(n-2)}{ 2 \hat{y}_C u \lrrb{2\hat{y}_C}} + \frac{u_y \lrrb{2\hat{y}_C}}{2}.$$
Therefore we get $u \lrrb{2 \hat{y}_C} \hat{y}_C \geq \frac{1}{D}$ for $C > 0$ small enough, after possibly adjusting $D > 0$. So it remains to show $\hat{y}_C \rightarrow 0$ as $C \rightarrow 0$ in order to get a contradiction.

Assume that $\hat{y}_C \nrightarrow 0$ as $C \rightarrow 0$, so without loss of generality there exists $\tilde{\varepsilon} > 0$ such that $\hat{y}_C \geq \tilde{\varepsilon}$ for any $C > 0$ small enough. As in the case $C \rightarrow \infty$ we get from $u_y \lrrb{\hat{y}_C} \rightarrow \infty$ that $\hat{y}_C u \lrrb{\hat{y}_C} \rightarrow 0$ as $C \rightarrow 0$. 

In view of the bound on $\frac{u \lrrb{\hat{y}_C}}{\hat{y}_C}$ and \ref{prop2} we see that $u \rightarrow 0$ uniformly on $\lrb{0,\tilde{\varepsilon}}$ as $C \rightarrow 0$ and 
$$ u \lrrb{\hat{y}_C} = u \lrrb{\frac{\hat{y}_C}{2}} + \int_{\frac{\hat{y}_C}{2}}^{\hat{y}_C} u_{\tilde{y}} d \tilde{y} \geq \frac{\hat{y}_C}{2} u_y \lrrb{\frac{\hat{y}_C}{2}}.$$
Therefore $u_y \lrrb{\frac{\hat{y}_C}{2}} \leq D$ and especially $u_y \leq D$ on $\lrb{0,\frac{\tilde{\varepsilon}}{2}}$, after possibly adjusting $D > 0$.

We can now apply Brakke's clearing out lemma (see appendix) for $M(u)$ (again with $u: \set{R} \rightarrow \set{R}$) with respect to the ball $B_{\frac{\tilde{\varepsilon}}{2}} \lrrb{0,C,0,\ldots,0}$ to get a contradiction. Estimating the Hausdorff measure yields, using our bound on $u_y$
$$ \mathcal{H}^{n-1} \lrrb{M(u) \cap B_{\frac{\tilde{\varepsilon}}{2}} \lrrb{0,C,0,\ldots,0}} \leq D(n) \tilde{\varepsilon} \max_{\lrb{0,\frac{\tilde{\varepsilon}}{2}}} u^{n-2}.$$
Since $\max_{\lrb{0,\frac{\tilde{\varepsilon}}{2}}} u^{n-2} \rightarrow 0$ as $C \rightarrow 0$ the clearing out time $t_C$ of $M(u)$ with respect to $B_{\frac{\tilde{\varepsilon}}{8}} \lrrb{0,C,0,\ldots,0}$ therefore satisfies $t_C \leq c \frac{\tilde{\varepsilon}^2}{4}$ for $C > 0$ small enough, where $c > 0$ is a constant. But $M(u)$ moves self-expanding, so $\sqrt{t_C + 1} C \geq C + \frac{\tilde{\varepsilon}}{8}$ which means $t_C \geq \frac{\tilde{\varepsilon}^2}{64 C^2}$ - a contradiction.

\end{proof}

Finally the next lemma is the desired stability result for (\ref{eq:selfexprotcurgraph}). 

\begin{lemma}
The asymptotic angle of solutions of (\ref{eq:selfexprotcurgraph}), $ \alpha = \alpha(C) : \; (0,\infty) \rightarrow \left(0,\frac{\pi}{2} \right)$, is a 
continuous function of the initial condition $ C > 0 $.
\end{lemma}
\begin{proof}
Assume not. Then there exists $ \tilde{\varepsilon} > 0, C_0 > 0 $ and a bounded sequence $ \left(C_k \right)_{k \in \set{N}} $ in $ (0,\infty) $ with 
$ C_k \rightarrow C_0 $ as $ k \rightarrow \infty $ and $ \abs{\tan \alpha_{C_k} - \tan \alpha_{C_0}} \geq \tilde{\varepsilon} $ for all $ k \in \set{N} $. The associated solutions of (\ref{eq:selfexprotcurgraph}) are denoted by 
$ \left(u_k \right)_{k \in \set{N}} $ and $ u_0 $.

From \ref{prop4} we see that $ 0 < \sup_{y \geq 0} \left( u_k \right)_y  \leq D$ for all $ k \in \set{N} $, $D > 0$ a constant. In the following we will adjust the constant $D > 0$ implicitely as necessary. In view of \ref{prop6} we assume that $ \left( u_0 \right)_{yy} $ has precisely one zero $ y_0 $. The case $ \left( u_0 \right)_{yy} > 0 $ on $ [0,\infty) $ can be handled in the same way.

Since we have continuous dependence on the initial conditions on compact intervals we may assume that all $ C_k $ are close enough to $ C_0 $, so that each $ \left(u_k \right)_{yy} $ has precisely one zero close to $ y_0 $. Therefore we get $ \sup_{y \geq 0} \left( u_k \right)_{yy} \leq D$ for all $ k \in \set{N} $. Differentiating (\ref{eq:selfexprotcurgraph}) with respect to $ y > 0 $ yields 
$$\frac{-(n-2)u_y}{u^2} = \frac{yu_{yy}}{2} + \frac{1}{1+\left( u_y \right)^2} \left( u_{y}^{(3)} - \frac{2u_y \left(u_{yy} \right) ^2}{1+\left( u_y \right)^2} \right).$$ 
In view of \ref{prop5} this means $ \left(u_k \right)_{yy} $ can not have arbitrarily small local extrema. Therefore 
$\abs{ \sup_{y \geq 0} \left( u_k \right)_{yy}} \leq D$ for all $ k \in \set{N} $.

Hence we can estimate for $ y > 0 $, using \ref{prop1} and \ref{prop6}
\begin{eqnarray}
\nonumber  \left( \frac{u_k}{y} \right)_y = \frac{1}{y} \left( \left( u_k \right)_y - \frac{u_k}{y} \right) = \\
\nonumber \frac{2}{y^2} \left( \frac{n-2}{u_k} - \frac{\left( u_k \right)_{yy}}{1+\left( u_y \right)^2} \right) \leq \frac{D}{y^2}.
\end{eqnarray}

So by \ref{prop3} for any $ \delta > 0 $ there exists $ y_{\delta} > 0 $ with
$$ \abs{ \frac{u_k}{y} - \tan \alpha_k } < \delta, \; \; \; \; y \geq y_{\delta}, \; k \in \set{N}. $$
Using our assumption on $\alpha{C_0}, \alpha_{C_k}$ we can therefore choose $ \delta > 0 $ such that 
$$ \abs{ \frac{u_k}{y} - \frac{u_0}{y} } \geq \frac{\tilde{\varepsilon}}{2}, \; \; \; \; y \geq y_{\delta}, \; k \in \set{N}.$$

We can write the difference of the ODEs for $ u_0 $ and $ u_k $ at $ y > 0 $ as $ \xi_1 = \xi_2 +  \xi_3 + \xi_4$, where

\begin{eqnarray}
\nonumber \xi_1 = \frac{1}{y} \left( u_0 - u_k \right) \\
\nonumber \xi_2 = \left(u_0 \right)_y - \left(u_k \right)_y \\
\nonumber \xi_3 = \frac{2}{y} \left( \frac{\left( u_0 \right)_{yy}}{1 + \left( u_0 \right)_y^2} - \frac{\left( u_k \right)_{yy}}{1 + \left( u_k \right)_y^2} \right) \\
\nonumber \xi_4 = \frac{2}{y} \left( \frac{n-2}{u_k} - \frac{n-2}{u_0} \right).
\end{eqnarray}

By the previous considerations we can find $\hat{y} > y_\delta$ such that $\xi_3 < \frac{\tilde{\varepsilon}}{6}$ and $\xi_4 < \frac{\tilde{\varepsilon}}{6}$  for any $k \in \set{N}$. By continuous dependence we can find $\hat{k} \in \set{N}$ such that $\xi_2 < \frac{\tilde{\varepsilon}}{6}$. Therefore we get  $\abs{\xi_2 +  \xi_3 + \xi_4} < \frac{\tilde{\varepsilon}}{2}$ - a contradiction.

\end{proof}

Putting the last three lemmas together we arrive at theorem \ref{thm:selfexp}. $\alpha_{crit}^* \in \lrrb{0, \frac{\pi}{2}}$ follows from \ref{prop5} and from the clearing out lemma (see appendix \ref{appendixclearingout} and \cite{angenentilmanen}). As stated in lemma \ref{lem:ilmanen1} the existence of a two-sheeted solution (for any cone angle $\alpha \in \lrrb{0, \frac{\pi}{2}}$) is asserted using results of Ecker and Huisken (see \cite{eckerhuisken}) which ensure existence of self-expanding evolutions of the two Lipschitz graphs $C_\alpha$ and $C_{\pi - \alpha}$.

Additionally standard theory of differential inequalities shows that for $C > \sqrt{2}$ the asymptotic angle $\alpha \lrrb{C}$ is strictly monotone  increasing for solutions of (\ref{eq:selfexprotcurgraph}). We believe it is everywhere strictly monotone, apart from $C_{crit}, \; \alpha \lrrb{C_{crit}} = \alpha_{crit}^*$. This would imply that for each $\alpha_{crit}^* < \alpha < \frac{\pi}{2}$ there are precisely two self-expanding. smooth evolutions of $D_\alpha$ of the form $M(u)$. In general there might be more self-expanding evolutions of $D_\alpha$, possibly also non-rotationally symmetric (see \cite{angenentilmanen}).

\begin{remark}
Solutions of (\ref{eq:selfexp1}) are stationary for the functional
$$ \mathbf{K}[M] = \int_M \exp \left( \dfrac{\abs{x}^2}{4} \right) d\mathcal{H}^{n-1}(x).$$
The authors of \cite{angenentilmanen} sketch a proof for the existence of one-sheeted self-expanders, asymptotic to $ D_\alpha $ which are minimizers of $ \mathbf{K} $. Furthermore they indicate how one might prove a version of theorem \ref{thm:selfexp} using this approach.
\end{remark}

\section{Touching fluid droplets}

As mentioned before we want to apply the previous results to study the behaviour of touching fluid droplets. These are assumed to have locally conical, 
rotationally symmetric shape, i.e. the shape of a smoothing of $ D_\alpha $. The following definition makes this formulation precise. From now on we set 
the dimension to $n=3$.

\begin{definition}
\label{def:conesmoothing}
We call $M_{\alpha}=M\left(u_\alpha \right)$
a smoothing of the double cone $ D_\alpha $ with angle $0<\alpha<\frac{\pi}{2}$,
$\gamma = \tan\alpha$, if for $a<0<b$
$$
u_\alpha\left(y \right)=\begin{cases}
\gamma \abs{y} \text{ if } y \leq a \text{ or }b\leq y\\
s \left( y \right) \textrm{ if }a\leq y \leq b\end{cases}
$$
such that $s>0$ and $u_\alpha \in C^{2,\beta}\left(\set{R}\right), \; \beta > 0$. $M_{\alpha}$ is said to lie outside of the double
cone $D_\alpha$ if $s \left( y \right) \geq\gamma \abs{y}$
for $a<y<b$. $ u $ is the generating function of $ M_\alpha $.
\end{definition}

Clearly any smoothing of a double cone stays rotationally symmetric under mean curvature flow. One can compute (see \cite{msimon}) that (\ref{eq:MCF}) 
for a mean curvature flow evolution $M_t$ of $M_\alpha $ is equivalent to

\begin{equation}
\label{eq:mcfgenfunc}
\frac{\partial}{\partial t}u=\frac{\frac{\partial^{2}}{\partial^{2}y}u}{1+\left(\frac{\partial}{\partial y}u\right)^{2}}-\frac{1}{u}
\end{equation}

where $u=u(\cdot,t), \; u(\cdot,0)=u_\alpha$ generates $M_{t}$.

For any smoothing $M_\alpha$ we have short-time existence of a solution $M_t$ of (\ref{eq:MCF}) on a maximal time interval $[0,T), \; T > 0$. Furthermore 
the solution must be smooth for $t > 0$ and every finite time singularity must be due to pinching, i.e. $\inf_{\set{R}} u(\cdot, t) \rightarrow 0$ as 
$t \rightarrow T$. This holds even without any growth assumption on the initial generating function (see \cite{msimon} and \cite{ladyzhenskayaparabolicpde}). 

In fact a sphere comparison argument (see \cite{eckermcf}) shows that
$$\lim_{t \rightarrow T} \min_{\set{R}} u(\cdot, t) \rightarrow 0 \text{ as } t \rightarrow T$$ 
must hold for finite time singularities. 

This agrees with intuition about repulsion (pinching in finite time) and coalescence (long-time existence) of fluid droplets.

Using techniques by Ecker and Huisken (see \cite{eckerhuisken}) one can derive global height estimates for rotationally symmetric solutions of the 
mean curvature flow which yield, using the results from \cite{barlesuniqueness}, the following comparison principle (see \cite{bodethesis}).

\begin{lemma}
\label{lem:compprinc}
Let $M_{\alpha_1}$, $M_{\alpha_2}$ be two smoothings of the double cone with $u_{\alpha_1}\leq u_{\alpha_2}$ for the associated
generating functions. Denoting the generating functions of the two evolutions with $u^1$ and $u^2$ we have then $u^1 \leq u^2$ as long as the solutions 
exist.
\end{lemma}

\begin{corollary}
Any smoothing of the double cone $M_\alpha, \; 0 < \alpha < \frac{\pi}{2}$ has a unique evolution by mean curvature.
\end{corollary}

Using these results we can now define what coalescence and repulsion mean within our model.

\begin{definition}
\label{def:repangle} 
An angle $0<\alpha<\frac{\pi}{2}$ is called a repulsion angle if there exists a smoothing of the double cone
$M_\alpha$ which is outside of the double cone and such that the mean curvature flow evolution of $M_\alpha$ pinches in finite time.
\end{definition}

\begin{lemma}
\label{lem:nonrep-coal}
An angle $0<\alpha<\frac{\pi}{2}$ is not a repulsion angle if and only if there is a smoothing of the cone $M_\alpha$, for which the evolution under 
mean curvature flow exists for all $t > 0$.
\end{lemma}
\begin{proof}
By definition any angle $0<\alpha<\frac{\pi}{2}$, that is not a repulsion angle, must have a smoothing $M_\alpha$ for which the evolution under mean 
curvature flow $M_t$ exists for all $t > 0$. 

So suppose for $0<\alpha<\frac{\pi}{2}$ there exists a smoothing $M_\alpha$, such that the evolution $M_t$ exists for all $t > 0$. Let $\hat{M}_\alpha$ be 
another smoothing, that is outside of the double cone. We denote its evolution by $\hat{M}_t, \; t \in [0,T)$.

We know that the mean curvature flow is invariant under parabolic rescaling
$$x\mapsto\lambda x, \; t\mapsto\lambda^{2}t $$
for any scaling parameter $\lambda>0$, $x \in M_{t}$ and $t\in I$. Let $M_{t}^{\lambda}$ be the rescaling of $M_{t}$. Note here that any double 
cone $D_\alpha$ is invariant under the scaling $x \mapsto\lambda x$.
Since $\hat{M}_\alpha$ is outside of the double cone we can therefore choose $\lambda>0$ sufficiently small in order to get initially
$$u^{\lambda}_\alpha\leq \hat{u}_\alpha $$
for the corresponding generating functions.
By lemma \ref{lem:compprinc} $\hat{M}_t$ must exist for all $t > 0$, therefore $\alpha$ is not a repulsion angle.
\end{proof}

In view of the last lemma we make the following definition.
\begin{definition}
\label{def:coalangle}
An angle $0<\alpha<\frac{\pi}{2}$ is called a coalescence angle if it is not a repulsion angle in the sense of definition \ref{def:repangle}.
\end{definition}

Given a field strength (respectively a cone angle) at which coalescence occurs any lower field strength (respectively greater cone angle)
should lead to coalescence as well. The same should hold for repulsion angles with higher field strength. The next two lemmas shows that this
is true within our model.

\begin{lemma}
\label{lem:justcoalangle}
Let $\alpha_{0}$ be a coalescence angle. Then for any angle $\alpha \geq \alpha_{0}$ and for any smoothing
$M_\alpha$ that is outside of the double cone the mean curvature flow with initial data $M_\alpha$ exists for all $t > 0$.
Therefore any angle $\alpha > \alpha_{0}$ is a coalescence angle.
\end{lemma}
\begin{proof}
As in the proof of lemma \ref{lem:nonrep-coal} we can scale appropriately and then use the comparison principle from lemma \ref{lem:compprinc}.
\end{proof}

\begin{lemma}
\label{lem:justrepangle}
Let $\alpha_{0}$ be a repulsion angle. Then for $\alpha\leq\alpha_{0}$ any smoothing of the double cone $M_\alpha$
with angle $\alpha$ must pinch in finite time. This means any angle $\alpha<\alpha_{0}$ must be a repulsion angle.
\end{lemma}
\begin{proof}
Again, as for lemma \ref{lem:nonrep-coal} we can scale appropriately and then use lemma \ref{lem:compprinc}.
\end{proof}

The observations from the experiments in \cite{naturephys} suggest the existence of a critical angle for the
behaviour of the system. As the next lemma shows, this is also true for our model.

\begin{lemma}
There is a critical angle $0 \leq\alpha_{crit}\leq \frac{\pi}{2}$ such that any angle $\alpha<\alpha_{crit}$ is a repulsion angle and any angle 
$\alpha>\alpha_{crit}$ is a coalescence angle.
\end{lemma}
\begin{proof}
Let 

\begin{eqnarray}
\nonumber \overline{\alpha}=\sup \lrbrace{0 <\alpha < \frac{\pi}{2}:\:\alpha\text{ is a repulsion angle}} \\
\nonumber \underline{\alpha}=\inf \lrbrace{ 0<\alpha< \frac{\pi}{2}:\:\alpha\text{ is a coalescence angle} }.
\end{eqnarray}
According to lemma \ref{lem:justrepangle} any cone angle $\alpha<\overline{\alpha}$
is not a coalescence angle. This shows that $\overline{\alpha}\leq\underline{\alpha}$.
Clearly any angle is either a repulsion or a coalescence angle, hence we have
$\overline{\alpha}=\underline{\alpha}=\alpha_{crit}$.
\end{proof}

Using the constructed one-sheeted self-expanders we can now determine $\alpha_{crit}$.

\subsection{Determining $\alpha_{crit}$}

\begin{lemma}
 \label{lem:upboundcrit}
Any double cone smoothing with angle $\alpha > \alpha_{crit}^*(3)$ that lies outside of the double cone has an evolution which exists for all $t > 0$.
\end{lemma}
\begin{proof}
Let $M_\alpha$ be a smoothing of $D_\alpha$ with $\alpha >  \alpha_{crit}^*(3)$ and generating function $u_\alpha$. Let $s$ be the generating function of the self-expander asymptotic to $D_{\alpha_{crit}^*(3)}$ which exists by theorem \ref{thm:selfexp}. As in the proof of lemma \ref{lem:nonrep-coal} we can do a parabolic rescaling and therefore assume $s \leq u_\alpha$. Then the evolution of $M_\alpha$ must exist for all $t > 0$ by the comparison principle, lemma \ref{lem:compprinc}, since every finite time singularity must be due to pinching.
\end{proof}

\begin{lemma}
 \label{lem:lowboundcrit}
Any double cone smoothing with angle $\alpha < \alpha_{crit}^*(3)$ must pinch in finite time.
\end{lemma}
\begin{proof}
 We follow a level-set flow argument from \cite{angenentilmanen}.

Assume that there is a double cone smoothing $M_\alpha$, $\alpha < \alpha_{crit}^*(3)$, such that its evolution by mean curvature $M_t$ exists for all $t > 0$. 

Let $\Gamma_t$ , $t \geq 0$, be the level-set flow of $D_\alpha$. $\Gamma_t$ can be characterized as follows. $\set{R}^3 \backslash \Gamma_t$ is the union of all level-set flows $\Delta_t$ such that $\Delta_0$ is compact and lies in $\set{R}^3 \backslash D_\alpha$. 

First we show that $0 \in \Gamma_t$. So let $\Delta_t$ be a level-set flow with $\Delta_0 \subset \set{R}^3 \backslash D_\alpha$ compact. $\Delta_0$ must either lie in the convex hull of one of $C_\alpha$ or $C_{\pi-\alpha}$ or outside of $D_\alpha$. Using the maximum principle for one level-set flow and one smooth flow, we see that in the first case $\Delta_0$ is pushed away from $0$ by the graphical self-expanders of Ecker and Huisken (part (i) of lemma \ref{lem:ilmanen1}) and therefore $0 \nelmt \Delta_t$. In the second case we can parabolically rescale $M_\alpha$ as in the proof of lemma \ref{lem:nonrep-coal} and therefore assume $\Delta_0 \cap M_\alpha = \emptyset$. Then we can apply the maximum principle to see that $0 \nelmt \Delta_t$. Hence by the above characterization we must have $0 \in \Gamma_t$.

$\Gamma_1$ is rotationally symmetric, so $\Gamma_1 = M(X)$ for some closed set $X \subset \set{R}^2$. In fact the boundary $\partial \Gamma_1$ is smooth and $\partial X$ consists precisely of curves of the type in lemma \ref{lem:ilmanen1} (see \cite{angenentilmanen}). 

Since $0 \in \Gamma_1$ there must be a curve of type (ii) in lemma \ref{lem:ilmanen1} in $\partial X$, which is asymptotic to $\sigma_\alpha \cup \sigma_{\pi - \alpha}$. This yields a smooth, rotationally symmetric, one-sheeted self-expanding evolution of $D_\alpha$ - a contradiction.
\end{proof}

In view of the definition of $\alpha_{crit}$ the last two lemmas yield the following

\begin{corollary}
 $\alpha_{crit} = \alpha_{crit}^*(3)$.
\end{corollary}

\subsection{Conclusions}

Using a model based on mean curvature flow, we obtain a critical cone angle $\alpha_{crit}$ for the behaviour of 
oppositely charged droplets of fluid. More precisely this means any smoothing of the double cone (see definition \ref{def:conesmoothing}) with angle
less than $\alpha_{crit}$ must pinch in finite time (repulsion). Assuming a smoothing has an angle greater than $\alpha_{crit}$ and lies outside of the 
double cone, its evolution must exist for all $t > 0$ (coalescence).

We can show that $\alpha_{crit} = \alpha_{crit}^*(3)$, where $\alpha_{crit}^*(3)$ is the critical angle for the existence of smooth, rotationally symmetric, self-expanding, one-sheeted evolutions of double cones and $\alpha_{crit}^*(3) \approx 66^\circ$.
This coincides with observations from experiments (see \cite{naturephys}) which predict a critical angle of $60^\circ-70^\circ$.

Finally we want to compare our model with the one in \cite{physrevletters}. For that approach it is assumed that the bridge between the touching droplets 
minimizes area under a volume constraint. This corresponds to constant mean curvature surfaces of revolution (Delaunay surfaces) which are fitted to linear 
double cones, similarly as in definition \ref{def:conesmoothing} (unlike in definition \ref{def:conesmoothing} the associated generating function is only continuous). The associated capillary 
pressure $p$ is assumed to determine the behaviour.

The critical shape has $p=0$ which yields a rescaled catenoid. Suitable smoothings of the associated surfaces with $p > 0$ provide double cone smoothings in the sense of definition \ref{def:conesmoothing} which pinch in finite time. This is in agreement with \cite{physrevletters}. However lemma \ref{lem:lowboundcrit} shows that down to a certain $p_0 < 0$ the (smoothings of the) associated surfaces with $p \leq 0$ still pinch in finite time under mean curvature flow. Therefore our predicted critical angle is greater than the predicted critical angle from \cite{physrevletters}, which is approximately $59^\circ$.

\appendix
\section{The clearing out lemma}
\label{appendixclearingout}
This version of the clearing out lemma is due to Brakke. It says that if the area ratio of a piece of a solution of mean curvature flow in a ball is initially small, then it must leave the ball with one quarter of the radius after a time proportional to the initial area ratio times the radius squared. Proofs can be found in \cite{brakkemcf} or \cite{eckermcf}.

\begin{appendixtheorem}
Let $\lrrb{M_t}_{t>0}$ be complete, properly embedded hypersurfaces in $\set{R}^n$, $n \geq 2$, evolving by mean curvature flow. There exist constants $c(n) > 0$ and $\varepsilon_0(n) > 0$ such that if $M_0$ satisfies 
$$\mathcal{H}^{n-1} \lrrb{M_0 \cap B_\rho \lrrb{x_0}} \leq \varepsilon_0 \rho^{n-1}$$
for some $x_0 \in \set{R}^n$ and $\rho > 0$, then
$$\mathcal{H}^{n-1} \lrrb{M_t \cap B_{\frac{\rho}{4}} \lrrb{x_0}} = 0$$
for $t = c \smallskip  \varepsilon_0 \rho^2$.
\end{appendixtheorem}

\section*{Acknowledgements}
The author is partially supported by The Leverhulme Trust.

\end{document}